%% file: signature_first_kind.tex
\documentclass[10pt,a4paper]{article}

\include{canevas_anglais}

\author{Nicolas Garrel}
\title{An Artin-Schreier-type theory for signatures of hermitian forms over involutions}
\date{}

\begin{document}

\maketitle

\section*{Introduction}

The theory of orderings of a field, as initiated by Artin and Schreier in \cite{AS},
has strong ties with the structure of the Witt ring $W(K)$ of the field, through
the study of signatures. Precisely, there is a canonical bijection between:
\begin{itemize}
\item the set of orderings $X(K)$ of a field $K$;
\item the signature maps $\sign_P: W(K)\to \Z$ at those ordering $P\in X(K)$;
\item the prime ideals of $W(K)$ or residual characteristic $0$.
\end{itemize}
The signature map $\sign_P$ sends a $1$-dimensional form $\langle a\rangle$ to
$\pm 1$ according to the $P$-sign of $a$, and the prime ideal $I_P(K)$ corresponding
to $P\in X(K)$ is simply the kernel of $\sign_P$ (we refer to \cite{Lam} for all
the details).
\\

A natural direction in which to generalize this framework is that of
$\eps$-hermitian forms over central simple algebras with involution,
with the Witt groups $W(A,\sigma)$ replacing $W(K)$. This was
initiated in \cite{BP}: given an ordering $P\in X(K)$, we can define
through a choice of Morita equivalence (see section \ref{sec_morita})
two possible group morphisms $W^\eps(A,\sigma)\to \Z$, which differ
by a sign and are still called signature maps at $P$.

One should note that this coincides with earlier notions of signatures of involutions
(developed for involutions of the first kind in in \cite{LT}, see also
\cite[A.11]{BOI}): only the absolute value of the signature of an involution
is properly defined, and for an $\eps$-hermitian form $h$, this applied
to the adjoint involution $\sigma_h$ does give the absolute value of the
signature of $h$, for either choice of sign. This sign ambiguity for
involutions can be seen as a manifestation of the fact that an involution
is only a descent of a bilinear form \emph{up to similitude}.

We would also like to point out that these definitions makes perfect sense for
symplectic involutions, following the usual pattern of symplectic involutions
being just as interesting as orthogonal ones even though in the split case they
are all hyperbolic. The equal status of orthogonal/symplectic involutions, or
equivalently hermitian/anti-hermitian forms, is an important feature of the
framework we develop in this article.
\\

In \cite{AU} and \cite{AU2}, Astier and Unger investigate those signature
maps, and determine criteria to coherently make those sign choices for all
orderings simultaneously. Their key result is
that one can find ``reference forms'' $h\in W(A,\sigma)$ such that
for any ordering $P$ of $K$, either all elements of $W(A,\sigma)$ have a zero
signature at $P$ (which is independent of the choice ambiguity), or $h$ has
a non-zero signature at $P$. One can then decide that $h$ should have a nonnegative
signature at all orderings, and this determines a pairing $W(A,\sigma)\times X(P)\to \Z$,
such that for any $x\in W(A,\sigma)$, its total signature map $X(K)\to \Z$ is
continuous. Of course, this depends on the choice of a reference form.

This can be pushed further to investigate signatures of $\eps$-hermitian forms
over algebras with involution. If $(V,h)$ is an $\eps$-hermitian form over $(A,\sigma)$,
it corresponds through hermitian Morita theory (see section \ref{sec_morita}) to
some $(B,\tau)$, which has a signature well-defined up to a sign at each ordering of
the base field. The question now becomes: given $(A,\sigma)$, how coherently can we
choose these signs, when both the hermitian form and the ordering vary?

If we fix the ordering, then this was addressed in \cite{BP}: 
essentially, a choice of hermitian Morita equivalence for $(A,\sigma)$ over
the real closure of the base field at the given ordering determines signs
simultaneously for all $\eps$-hermitian forms, and a different choice flips
all the signs. So there are two opposite coherent choices when fixing the
ordering, each giving an additive map $W^\eps(A,\sigma)\to \Z$.

In \cite{AU} and \cite{AU2}, Astier and Unger investigate how to coherently
choose this additive map for all orderings simultaneously. Their key result is
that one can find ``reference forms'' $h\in W(A,\sigma)$ such that
for any ordering $P$ of $K$, either all elements of $W(A,\sigma)$ have a zero
signature at $P$ (which is independent of the choice ambiguity), or $h$ has
a non-zero signature at $P$. One can then decide that $h$ should have a nonnegative
signature at all orderings, and this determines a pairing $W(A,\sigma)\times X(P)\to \Z$,
such that for any $x\in W(A,\sigma)$, its total signature map $X(K)\to \Z$ is
continuous. Of course, this depends on the choice of a reference form.
\\

It would be desirable to wrap all those investigations in an Artin-Schreier-type
theory. Unfortunately, the Witt group $W(A,\sigma)$ does not carry a natural ring
structure, which prevents a straightforward generalization in this setting. This
is somewhat circumvented in \cite{AU2}, where the authors introduce an ad hoc class of
morphisms $W(A,\sigma)\to \Z$ which are to be interpreted as ``ring morphisms'', and
do correspond to signature maps, but this is not fully satisfactory.

In \cite{G}, we introduced a ring $\tld{W}(A,\sigma)$, called the \emph{mixed Witt ring}
of an algebra with involution of the first kind. As a group, it is simply
\[ \tld{W}(A,\sigma) = W(K)\oplus W(A,\sigma) \oplus W^{-1}(A,\sigma). \]
The thesis of this article is that this ring constitutes an appropriate
framework to transpose the Artin-Schreier theory. Our main results are the following:

\begin{thm}
  For any ordering $P\in X(K)$, there are exactly two ring morphisms
  $\tld{W}(A,\sigma)\to \Z$ extending the signature $W(K)\to \Z$, differing
  by a sign on $W^{\pm 1}(A,\sigma)$. The set $\tld{X}(A,\sigma)$ of their kernels
  is exactly the set of prime ideals of $\tld{W}(A,\sigma)$ of residual
  characteristic $0$.
\end{thm}

Our point of view is that $\tld{X}(A,\sigma)$ is the natural replacement
of $X(K)$ in this context. A coherent choice of signs as alluded above
amounts exactly to a choice of section of the obvious map $\tld{X}(A,\sigma)\to X(K)$
(what we call a \emph{polarization} in section \ref{sec_polar}).

\begin{thm}
  If we endow $\tld{X}(A,\sigma)$ with the Zariski topology, the natural map
  $\tld{X}(A,\sigma)\to X(K)$ is a topological double cover,
  and the polarizations defined by Astier and Unger in terms of reference forms
  are exactly the continuous sections.
\end{thm}

Note that while we are focusing here on involutions of the first kind, the case of
involutions of the second kind has also received attention (in \cite{AQM} for involutions,
and \cite{AU} and \cite{AU2} for hermtian forms), and we intend to give it
a similar treatment in a future article.

\section*{Notations}

In all this article, $K$ is a fixed field of characteristic not $2$.

Its Witt ring is $W(K)$, with fundamental ideal $I(K)$. We will
not distinguish between quadratic forms and their Witt classes. The diagonal
quadratic forms are denoted $\fdiag{a_1,\dots,a_n}$, and the
$n$-fold Pfister forms $\pfis{a_1,\dots,a_n}$. Recall that in particular
the $n$-fold Pfister form $\pfis{-1,\dots,-1}$ is the element $2^n\in W(K)$.

The map $e_2: I^2(K)\to H^2(K,\mu_2)$ is the Clifford invariant map.
It sends a 2-fold Pfister form $\pfis{a,b}$ to the Brauer class $(a,b)$.
In particular, if $Q$ is a quaternion algebra, its norm form $n_Q$ can
be characterized by the fact that $e_2(n_Q)$ is the Brauer class $[Q]$ of $Q$.

An algebra with involution $(A,\sigma)$ over $K$ is a central simple
$K$-algebra of finite dimension, and $\sigma$ is an involution of the first kind,
so $\sigma$ restricts to the identity on $K$. Recall that $\sigma$ can
be either orthogonal or symplectic. The reduced trace map of $A$
is $\Trd_A: A\to K$.

We write $W^\eps(A,\sigma)$ for the Witt group of regular $\eps$-hermitian
forms over $(A,\sigma)$, with $\eps\in \{+1,-1\}$. We often write
$W(A,\sigma)$ for $W^1(A,\sigma)$. If $a_1,\dots,a_n\in A^\times$ are
$\eps$-symmetric with respect to $\sigma$ (meaning that $\sigma(a_i)=\eps a_i$),
then we write $\fdiag{a_1,\dots,a_n}_\sigma\in W^\eps(A,\sigma)$ for the
corresponding diagonal $\eps$-hermitian form. In particular, we have a
canonical element $\fdiag{1}_\sigma\in W(A,\sigma)$.

If $\sigma$ is orthogonal then we set
$W_\eps(A,\sigma) = W^\eps(A,\sigma)$; if $\sigma$ is symplectic then
$W_\eps(A,\sigma) = W^{-\eps}(A,\sigma)$. 

For any commutative ring $R$, $\Spec_0(R)\subset \Spec(R)$ is the
subset of prime ideals over $(0)\in \Spec(\Z)$. We view it as a topological
subspace.

\section{Orderings and signatures over fields}\label{sec_field}

We start with a brief overview of the theory over fields, mainly to fix
notation, and we refer to \cite{Lam} for proofs.

Recall that a field $K$ is formally real if $-1$ is not a sum of squares in
$K$; it is real closed if in addition no algebraic extension of $K$
is formally real.  
An ordering on a field $K$ is a subgroup $P\subset K^*$ of index 2 which
is stable under addition and does not contain $-1$. We write $X(K)$ for the
set of all orderings of $K$. If $P\in X(K)$ we say that
$(K,P)$ is an ordered field, and we write $\sign_P:K^\times\to \{\pm 1\}$
the morphism with kernel $P$. Then $\sign_P(a)$ is called the $P$-sign
(or the sign if no confusion is possible) of $a\in K^*$.
We can then speak of $P$-positive and $P$-negative elements.

An extension of an ordered field $(K,P)$ is an ordered field $(L,Q)$
such that $L/K$ is an extension with $P=Q\cap K$. If $L/K$ is algebraic
and $(L,Q)$ is real closed, then $(L,Q)$ is called a real closure of $K$.

\begin{prop}
  A field $K$ admits an ordering iff it is formally real; a real closed
  field admits a unique ordering. Any ordered field $(K,P)$ admits a real closure
  $K_P$, unique up to a unique $K$-isomorphism.
\end{prop}

\begin{prop}
  If $L$ is real closed, then there is a (unique) ring isomorphism
  between $W(L)$ and $\Z$, sending $\fdiag{a}$ to its sign relative to the
  unique ordering of $L$, for all $a\in L^*$.
\end{prop}

Thus for any ordering $P$ on a field $K$, there is a unique ring morphism
\[ \sign_P: W(K)\To W(K_P) \Isom \Z, \]
called the \emph{signature} of $K$ at $P$, which extends the
$P$-sign map on $K^*$ (meaning that $\sign_P(\fdiag{a})=\sign_P(a)$).
For any $p\in \N$ that is either $0$ or a prime number, we write
\[ \sign_{P,p}: W(K) \xrightarrow{\sign_P} \Z \To \Z/p\Z. \]
Then we set
\[ I_{P,p}(K) = \Ker(\sign_{P,p}). \]
We also write $I_P(K)=I_{P,0}(K)$.

\begin{prop}\label{prop_spec_w}
  For any ordering $P$ on $K$, we have $I_{P,2}(K)=I(K)$.
  Furthermore,
  \[ \Spec(W(K)) = \{I(K)\} \coprod_{P\in X(K)} \ens{I_{P,p}(K)}{\text{$p$ odd or $0$}}. \]
  In particular, there is a canonical identification between $X(K)$
  and $\Spec_0(W(K))$.
\end{prop}

\begin{rem}
  The topology given on $X(K)$ through this identification is the
  so-called \emph{Harrison topology}. It is well-known that $X(K)$
  is compact Hausdorff and totally disconnected, and the embedding
  $X(K)\to \{\pm 1\}^{K^*}$ given by $P\mapsto \sign_P$ is actually
  a closed immersion.
\end{rem}

\section{Hermitian Morita theory and the mixed Witt ring}\label{sec_morita}

In this section we review the necessary material from \cite{G}, starting with
some hermitian Morita theory.

If $(A,\sigma)$ and $(B,\tau)$ are algebras with involution over $K$,
a hermitian Morita equivalence from $(B,\tau)$ to $(A,\sigma)$ is
a $B$-$A$-bimodule $V$ such that $B\simeq \End_A(V)$, endowed with
an $\eps$-hermitian form $h:V\times V\to A$ such that $\tau$ is the
adjoint involution of $h$.

Such an equivalence exists if and only if $A$ and $B$ are Brauer-equivalent,
and in that case $V$ is unique up to bimodule isomorphism, while
$h$ is unique up to multiplication by $\fdiag{\lambda}$ for any
$\lambda\in K^\times$. Then $\eps=1$ if $\sigma$ and $\tau$ have the
same type, and $\eps=-1$ if they have opposite type.

Any hermitian Morita equivalence $(V,h)$ induces a $W(K)$-module isomorphism
$W_\eps(B,\tau)\isom W_\eps(A,\sigma)$ for each $\eps\in \{\pm 1\}$, such that
$\fdiag{1}_\tau$ is sent to $h$.

In \cite{G} we define a category $\CBrh$, called the hermitian Brauer 2-group
of $K$, such that the objects are the algebras with involution $(A,\sigma)$,
and the morphisms $(B,\tau)\to (A,\sigma)$ (which are all invertible)
are the hermitian Morita equivalences.
From what was explained above, if we consider the $W(K)$-modules
\[ \tld{W}_\eps(A,\sigma) = W(K) \oplus W_\eps(A,\sigma) \]
and
\[ \tld{W}(A,\sigma) = W(K) \oplus W_1(A,\sigma) \oplus W_{-1}(A,\sigma), \]
then $\tld{W}$ and $\tld{W}_\eps$ are functors from $\CBrh$ to the category
of $W(K)$-modules. Note that morally $\tld{W}(A,\sigma)$ should have four factors,
but $W^{-1}(K)=0$. Then the key result of \cite{G} is:

\begin{prop}
  The natural $W(K)$-module structure on $\tld{W}(A,\sigma)$
  extends to a $W(K)$-algebra structure, such that $\tld{W}$ is a functor
  from $\CBrh$ to the category of commutative $W(K)$-algebras.

  It satisfies $W_1(A,\sigma)\cdot W_{-1}(A,\sigma)=0$, and
  $W_\eps(A,\sigma)\cdot W_\eps(A,\sigma)\subset W(K)$, so in particular
  $\tld{W}_\eps(A,\sigma)$ is a subring.

  Furthermore, for any field extension $L/K$, the natural map
  $\tld{W}(A,\sigma)\to \tld{W}(A_L,\sigma_L)$ is a ring morphism.
\end{prop}

We will need an explicit description of this structure in two fundamental cases:

\begin{ex}
  If $(A,\sigma)=(K,\Id)$, then $\tld{W}(K,\Id)= W(K)\oplus W(K)$ is naturally
  isomorphic as a ring to the group algebra $W(K)[\Z/2\Z]$.
\end{ex}

\begin{ex}\label{ex_witt_quater}
  If $(A,\sigma)=(Q,\gamma)$ is a quaternion algebra with its canonical
  symplectic involution, then we have an explicit description of the product
  in $\tld{W}(Q,\gamma)$:
  \begin{itemize}
  \item for all $a,b\in K^\times$,
    \[ \fdiag{a}_\gamma\cdot \fdiag{b}_\gamma = \fdiag{2ab}n_Q. \]
  \item for all pure quaternions $z_1,z_2\in Q^\times$,
    \[ \fdiag{z_1}_\gamma\cdot \fdiag{z_2}_\gamma = \fdiag{-\Trd_Q(z_1z_2)}\phi_{z_1,z_2} \]
    where $\phi_{z_1,z_2}$ is the unique $2$-fold Pfister form such that
    $e_2(\phi_{z_1,z_2})= (z_2^2,z_2^2)+ [Q]$. Note that $\Trd_Q(z_1z_2)$ might
    be zero, but in that case $z_1$ and $z_2$ anti-commute, so $\phi_{z_1,z_2}$
    is actually hyperbolic anyway (and the product is zero).
  \end{itemize}
\end{ex}

\begin{rem}\label{rem_trace}
  We will not need to know any more about the ring structure of $\tld{W}(A,\sigma)$
  for the purpose of this article, but it is useful to have in mind that
  $\fdiag{1}_\sigma^2\in W(K)$ is nothing else than the so-called involution
  trace form $T_\sigma$. Explicitly, the quadratic form $T_\sigma: A\to K$ is defined
  by $a\mapsto \Trd_A(\sigma(a)a)$.

  Together with the functoriality properties, this implies that for any $h\in W_\eps(A,\sigma)$,
  we have $h^2 = T_{\sigma_h}\in W(K)$, where $\sigma_h$ is the adjoint involution of $h$.
\end{rem}

\section{Signature maps}\label{sec_signature}

In this section we describe all signature maps:

\begin{defi}
  Let $(A,\sigma)$ be an algebra with involution over $K$, and $P\in X(K)$
  be an ordering of $K$. A signature map of $(A,\sigma)$ at $P$ is
  a ring morphism $\tld{W}(A,\sigma)\to \Z$ which extends the classical
  signature map $W(K)\to \Z$ at $P$.
\end{defi}

There is a situation where such maps clearly exist:

\begin{defi}\label{def_retraction}
  Let $(A,\sigma)$ be an algebra with involution over $K$.
  We say that a ring morphism $\rho: \tld{W}(A,\sigma)\to W(K)$ is a
  retraction of $\tld{W}(A,\sigma)$ if it is the identity on $W(K)$.
\end{defi}

\begin{ex}\label{ex_retrac_ortho}
  The augmentation map $W(K)[\Zd]\to W(K)$ defines a retraction $\rho$ of
  $\tld{W}(K,\Id)$, which we call the canonical retraction of $\tld{W}(K,\Id)$.
\end{ex}

Using simple composition, the choice of a retraction (if it exists)
allows to define signature maps at all $P\in X(K)$ simultaneously;
see Section \ref{sec_polar} for further discussion on this phenomenon, which we
call there an \emph{algebraic polarization}.

\begin{rem}
  The existence of a retraction of $\tld{W}(A,\sigma)$ only depends
  on the Brauer class $[A]$, since if $[B]=[A]$ then for any
  involution $\tau$ on $B$ there is a $W(K)$-algebra isomorphism
  $\tld{W}(B,\tau)\simeq \tld{W}(A,\sigma)$.
\end{rem}

Following the method of the classical case (see section \ref{sec_field}),
our general strategy to define signature maps is to reduce to
the case of real closed fields. For those cases, we will show that
canonical retractions always exist. Of course, over a real closed
field, since $W(K)$ is canonically isomorphic to $\Z$, a retraction
is really the same thing as a signature map, but there is some
conceptual clarity (and some additional generality) gained by
expressing this in terms of retractions.

Recall that a classical result (due to Frobenius when $K=\mathbb{R}$) states
that if $K$ is real closed, the only division algebras over $K$ are $K$ itself
and the Hamilton quaternions $\mathbb{H}_K$, which is the quaternion algebra
with Brauer class $[\mathbb{H}_K] = (-1,-1)\in H^2(K,\mu_2)$. This
implies that any algebra with involution over $K$ is Morita equivalent to either
$(K,\Id)$ or $(\mathbb{H}_K,\gamma)$. We already saw in example \ref{ex_retrac_ortho}
that $\tld{W}(K,\Id)$ admits a canonical retraction, so it just
remains to treat the case of $(\mathbb{H}_K,\gamma)$.

We do this in a slightly more general setting. Recall that a field $K$
is \emph{Pythagorean} if a sum of squares in $K$ is a square (we refer
to \cite{Lam} for more details). In that case, either $K$ is quadratically
closed, or $K$ is formally real and $W(K)$ is torsion-free as an abelian group.
Of course a real closed field is Pythagorean.

\begin{prop}\label{prop_retrac_sympl}
  Let $K$ be a formally real Pythagorean field.
  Then there exist retractions $\rho: \tld{W}(\mathbb{H}_K,\gamma)\to W(K)$,
  and there is a unique one, called the canonical retraction, such
  that $\rho(\fdiag{z}_\gamma)=0$ for any invertible pure quaternion
  $z\in \mathbb{H}_K$, and $\rho(\fdiag{1}_\gamma)=2$.
\end{prop}

\begin{proof}
  We define $\rho$ as the identity on $W(K)$ and the zero function
  on $W^{-1}(\mathbb{H}_K,\gamma)$. It remains to define $\rho(h)$
  for $h\in W(\mathbb{H}_K,\gamma)$. 

  In general, for any quaternion algebra $Q$ over any field $K$, if
  $(V,h)\in W(Q,\gamma)$, then there is a natural quadratic
  form $q_h:V\to K$ defined by $q_h(x)=h(x,x)$. We claim that
  in our case, $q_h=2\rho(h)$ for some $\rho(h)\in W(K)$, which is
  then unique since $W(K)$ is torsion free.

  Indeed, in general if $h=\fdiag{a_1,\dots,a_n}_\gamma$ then
  $q_h=\fdiag{a_1,\dots,a_n}n_Q$. In our case, $n_{\mathbb{H}_K}=\pfis{-1,-1}=4\in W(K)$,
  so we just have to take $\rho(h)=2\fdiag{a_1,\dots,a_n}$. The
  way we introduced $\rho(h)$ shows that this does not depend on the
  diagonalization.

  By construction, $\rho: \tld{W}(\mathbb{H}_K,\gamma)\to W(K)$ is then
  a $W(K)$-module morphism, and it is the only one satisfying the
  conditions of the proposition, since $W(\mathbb{H}_K,\gamma)$ is
  generated by $\fdiag{1}_\gamma$ as a $W(K)$-module. It remains
  to show that $\rho$ is a ring morphism, which amounts to $\rho(xy)=\rho(x)\rho(y)$
  for $x\in W^\eps(\mathbb{H}_K,\gamma)$ and $y\in W^{\eps'}(\mathbb{H}_K,\gamma)$.
  
  If $\eps\neq \eps'$, then $xy$ and $\rho(y)$ are zero, thus we can
  assume that $\eps=\eps'$. If $\eps=1$, we can reduce to $x=\fdiag{a}_\gamma$
  and $x=\fdiag{b}_\gamma$ for $a,b\in K^\times$, and according to example
  \ref{ex_witt_quater}, we have to show that
  \[ \fdiag{2ab}n_{\mathbb{H}_K} = (2\fdiag{a})\cdot (2\fdiag{b}), \]
  which holds since $n_{\mathbb{H}_K}=4\in W(K)$ and $\fdiag{2}$
  is represented by $n_{\mathbb{H}_K}$.

  If $\eps=-1$, we have to show that for all invertible pure quaternions
  $z_1,z_2\in \mathbb{H}_K$,
  \[ \fdiag{-\Trd_Q(z_1z_2)}\phi_{z_1,z_2} = 0, \]
  which means that $\phi_{z_1,z_2}$ is always hyperbolic, and
  by definition this means
  \[ (z_1^2,z_2^2) = (-1,-1). \]
  But for any pure quaternion $z$, $z^2$ is the opposite of a sum
  of three squares, so since $K$ is Pythagorean, $z^2$ is the opposite
  of a square.
\end{proof}

We now come back to the general case. Let $(A,\sigma)$ be
an algebra with involution over an arbitrary field $K$. For
any $P\in X(P)$, let $D_P$ be the unique division algebra over
$K_P$ Brauer-equivalent to $A_{K_P}$. We define an involution
$\theta_P$ on $D_P$, as well as a partition $X(K) = X_1(A)\coprod X_{-1}(A)$,
through the following dichotomy:
\begin{itemize}
\item if $D_P=K_P$, then $\theta_P=\Id$ and $P\in X_1(A)$;
\item if $D_P=\mathbb{H}_{K_P}$, then $\theta_P=\gamma$ and $P\in X_{-1}(A)$.
\end{itemize}

\begin{ex}
  If $A$ is split, $X_1(A)=X(K)$ and $X_{-1}(A)=\emptyset$.
  On the other hand, $X_1(\mathbb{H}_K)=\emptyset$ and
  $X_{-1}(\mathbb{H}_K)=X(K)$. 
\end{ex}

\begin{rem}
  The point of this partition is that signature maps
  at $P\in X_\eps(A)$ will be zero on $W_{-\eps}(A,\sigma)$.
  Since we give equal status to $W_1(A,\sigma)$ and
  $W_{-1}(A,\sigma)$, our treatment of $X_1(A)$ and $X_{-1}(A)$
  is symmetric.

  In \cite{AU}, the authors only study signatures on $W(A,\sigma)$,
  which creates an asymmetry. In their terminology, $Nil[A,\sigma]$
  is $X_1(A)$ if $\sigma$ is symplectic, and $X_{-1}(A)$ if $\sigma$ is
  orthogonal (so $Nil[A,\sigma]$ is always the set of ``uninteresting''
  orderings for $W(A,\sigma)$).
\end{rem}

For any $P\in X(K)$, there are exactly two isomorphisms
$(A_{K_P},\sigma_{K_P})\to (D_P,\theta_P)$ in $\mathbf{Br}_h(K_P)$.
Indeed, there is at least one since by construction $[A_{K_P}]=[D_P]$,
and any two choices differ by the multiplication by some $\fdiag{a}$
with $a\in K_P^\times$; but since $K_P$ is real closed, it has only two
square classes. Let us write $f_P^\eta$ for those two isomorphisms, with
$\eta=\pm 1$; choosing which one has label $\eta=1$ or $\eta=-1$ is arbitrary.
A simultaneous choice for all $P\in X(P)$ is what we will later call a polarization
(see section \ref{sec_polar}). Until then we always assume that we made
an arbitrary but fixed choice for all $P$.

For any $P\in X(K)$, we have a canonical retraction
$\rho_P: \tld{W}(D_P,\theta_P) \to W(K_P)$, provided either by example \ref{ex_retrac_ortho}
or proposition \ref{prop_retrac_sympl}. We can then define two signature maps of $(A,\sigma)$
at $P$ by:
\[ \tld{\sign}_P^\eta: \tld{W}(A,\sigma)\To \tld{W}(A_{K_P},\sigma_{K_P})
  \xrightarrow{(f_P^\eta)_*} \tld{W}(D_P,\theta_P) \xrightarrow{\rho_P} W(K_P) \xrightarrow{\sign} \Z. \]

We need a technical result from Astier and Unger:

\begin{lem}[\cite{AU}, thm 6.1]\label{lem_h_local}
  Let $\eps,\eta\in \{\pm 1\}$ and let $P\in X_\eps(A)$. Then there exists
  $h\in W_\eps(A,\sigma)$ such that $\tld{\sign}_P^\eta(h)\neq 0$.
\end{lem}

We can now check that our construction is exhaustive:

\begin{prop}\label{prop_signatures}
  Let $(A,\sigma)$ be an algebra with involution over $K$. For any
  $P\in X(K)$, there are exactly two distinct signature maps of
  $(A,\sigma)$ at $P$, namely the $\tld{\sign}_P^\eta$ for $\eta=\pm 1$.
  If $P\in X_\eps(A)$, then these two maps are $0$ on $W_{-\eps}(A,\sigma)$,
  and they differ by a sign on $W_\eps(A,\sigma)$.
\end{prop}

\begin{proof}
  The fact that both $\tld{\sign}_P^\eta$ are zero on $W_{-\eps}(A,\sigma)$
  follows directly from the construction of the retraction $\rho_P$
  in each case. The fact that they differ by a sign on $W_\eps(A,\sigma)$
  follows from the fact that $f_P^1$ and $f_P^{-1}$ differ by multiplication
  by $\fdiag{-1}=-1\in W(K)$. Lemma \ref{lem_h_local} then shows that the
  $\tld{\sign}_P^\eta$ are two different functions.
  
  Let $f:\tld{W}(A,\sigma)\to \Z$ be a ring morphism extending
  $\sign_P$, and let $x\in W_{\eps'}(A,\sigma)$ for some $\eps'=\pm 1$.
  Then $f(x)^2 = \sign_P(x^2) = \tld{\sign}_P^\eta(x)^2$ (for any $\eta$),
  so $f(x)=\tld{\sign}_P^{s(x)}(x)$ for some $s(x)=\pm 1$ (which is
  uniquely determined if $f(x)\neq 0$).
  Now we want to show that we can take the same $s(x)$ for all
  possible $x$.
  
  If $\eps'=-\eps$, then $f(x)=\tld{\sign}_P^\eta(x)=0$ for
  both $\eta$, so we can choose $s(x)$ arbitrarily for all
  $x\in W_{-\eps}(A,\sigma)$.
  
  If $\eps'=\eps$, according to lemma \ref{lem_h_local}, there is some
  $y\in W_{\eps}(A,\sigma)$ such that $\tld{\sign}_P^\eta(y)\neq 0$;
  in particular, $s(y)$ is uniquely determined, and we call it $s$.
  Now for any $x\in W_{\eps}(A,\sigma)$, if $f(x)= 0$ then we can choose $s(x)=s$
  if we want, and if $f(x)\neq 0$, then $s(x)$ is uniquely determined and
  \[ f(xy) = \sign_P(xy) = \tld{\sign}_P^s(xy)
       = s(x)s\cdot \tld{\sign}_P^{s(x)}(x)\tld{\sign}_P^s(y) = s(x)s\cdot f(x)f(y) \]
  and since $f(x)f(y)\neq 0$, we have $s(x)s=1$.
\end{proof}

\begin{rem}
  Looking at the proof, we see that the only ring morphisms
  $\tld{W}_\eps(A,\sigma)\to \Z$ extending the signature at $P$ are
  the restrictions of the signature maps $\tld{\sign}_P^\eta$ (those
  two restrictions coincide exactly when $P\in X_{-\eps}(A)$).
\end{rem}

\begin{rem}\label{rem_compar_sign_au}
  We warn the reader that while our signature maps are
  essentially the same as those in previous literature, they
  are normalized differently when $P\in X_{-1}(A)$. Indeed,
  previously signature maps where only required to be $W(K)$-module
  morphism $W(A,\sigma)\to \Z$, where $\Z$ is seen as a $W(K)$-module
  through $\sign_P$. In that case, Astier and Unger show in
  \cite[prop 7.4]{AU2} that one can find a ``minimal'' surjective signature
  such that any other module morphism is an integer multiple
  of this minimal one (which is then well-defined up to a sign,
  of course), and they take their signatures to be those minimal
  ones. This gives the same definition as ours if $P\in X_1(A)$,
  but if $P\in X_{-1}(A)$ we have to multiply those minimal
  maps by $2$ to get a ring morphism. 
\end{rem}

\begin{rem}
  We can also define a signature of the involution $\sigma$:
  $\sign_P^\eta(\sigma)\stackrel{\text{def}}{=}\tld{\sign}_P^\eta(\fdiag{1}_\sigma)$.
  We again encounter a sign ambiguity, which is why in \cite{LT}
  only the absolute value $|\sign_P^\eta(\sigma)|$ is defined,
  and taken as the definition of the signature of $\sigma$.
  Note that the definitions agree since they characterize
  $\sign_P(\sigma)\in\N$ by $\sign_P(\sigma)^2=\sign_P(T_\sigma)$,
  and $T_\sigma=\fdiag{1}_\sigma^2$ in $\tld{W}(A,\sigma)$ (see remark \ref{rem_trace}).
\end{rem}

\section{The spectrum of the mixed Witt ring}

Now that we have our signature maps, we want to obtain a
description of $\Spec(\tld{W}(A,\sigma))$ similar to proposition
\ref{prop_spec_w} for $W(K)$. We achieve this by studying the fibers
of the natural morphism
\[ \pi:  \Spec(\tld{W}(A,\sigma))\To \Spec(W(K)) \]
induced by the inclusion.
\\

First we need the equivalent of the fundamental ideal.
Recall that if $V$ is a (right) $A$-module, its \emph{reduced dimension}
is $\rdim(V)=\frac{\dim_K(V)}{\deg(A)}$, or equivalently it is the degree
of $\End_A(V)$. Then since hyperbolic forms have an even reduced dimension,
the Witt class of an $\eps$-hermitian space $(V,h)$ has a well-defined
reduced dimension mod $2$: $\rdim_2(h)\in \Zd$. This defines additive
morphisms $W^\eps(A,\sigma)\to \Zd$, and in addition to the classical
``dimension mod $2$'' morphism $W(K)\to \Zd$, we get a map
\[ \rdim_2: \tld{W}(A,\sigma)\to \Zd \]
which is easily seen to be a ring morphism.

\begin{defi}
  The fundamental ideal of $\tld{W}(A,\sigma)$ is $I(A,\sigma)=\Ker(\rdim_2)$.
\end{defi}

\begin{rem}
  If $A$ is not split, then $I(A,\sigma)=I(K)\oplus W(A,\sigma)\oplus W^{-1}(A,\sigma)$.
  If $A$ is split, then under any Morita equivalence $\tld{W}(A,\sigma)\simeq W(K)[\Zd]$,
  we have $I(A,\sigma)=\ens{x,y\in W(K)}{x+y\in I(K)}$.
\end{rem}

Let $P\in X(K)$ and let $p\in \N$ be either 0 or a prime number.
Then we define for $\eta\in \{\pm 1\}$:
\[ \tld{\sign}_{P,p}^\eta: \tld{W}(A,\sigma)\xrightarrow{\tld{\sign}_P^\eta} \Z\To \Z/p\Z \]
and
\[ I_{P,p}^\eta(A,\sigma) = \Ker(\tld{\sign}_{P,p}^\eta), \]
which is by construction a prime ideal of $\tld{W}(A,\sigma)$
(maximal if $p\neq 0$).


\begin{prop}\label{prop_spec_w_mixte}
  Let $(A,\sigma)$ be an algebra with involution over $K$. Then
  the fiber of $\pi$ above $I(K)$ is $\{I(A,\sigma)\}$. In particular,
  $I(A,\sigma)$ is the only prime ideal of $\tld{W}(A,\sigma)$
  with residual characteristic 2, and for any $P\in X(K)$,
  we have $I_{P,2}^\eta(A,\sigma) = I(A,\sigma)$.

  Furthermore, let $P\in X(K)$, and $p$ be either $0$ or an odd prime.
  Then the fiber of $\pi$ above $I_{P,p}(K)$ is
  $\ens{I_{P,p}^\eta(A,\sigma)}{\eta=\pm 1}$ (the two being distinct).

  

\end{prop}

\begin{proof}
  Let $I\subset \tld{W}(A,\sigma)$ be a prime ideal with residual
  characteristic 2. We want to show that $I(A,\sigma)= I$, which also
  implies the statement about $I_{P,2}^\eta(A,\sigma)$ and about the fiber
  of $\pi$ above $I(K)$. Let $x=x_0+x_1$ with $x_0\in W(K)$ and
  $x_1\in W(A,\sigma)\oplus W^{-1}(A,\sigma)$. It is enough to show that
  $x^2\in I$ iff $x^2\in I(A,\sigma)$. Now $x^2=x_0^2+2x_0x_1+x_1^2$,
  and $2x_0x_1$ is always in $I$ and $I(A,\sigma)$ since they have residual
  characteristic $2$, so we are reduced to show that $y=x_0^2+x_1^2\in W(K)$
  is in $I$ iff it is in $I(A,\sigma)$. But since $I\cap W(K)$ has residual
  characteristic $2$, $I\cap W(K)=I(K)=I(A,\sigma)\cap W(K)$.

  Now let $P\in X(K)$, and $p$ be either $0$ or an odd prime;
  we set $R=\Z/p\Z$. Let $I$ be in the fiber of $\pi$ above
  $I_{P,p}(K)$, and let $f: \tld{W}(A,\sigma)\to S$ be the
  surjective morphism with kernel $I$, with $R\subset S$. Then
  since $f_{|W(K)}=\sign_P$,
  the same proof as for proposition \ref{prop_signatures}
  shows that $R=S$ and $f=\tld{\sign}_{P,p}^\eta$ for
  a unique $\eta=\pm 1$. Indeed, we show the same way that for fixed
  $x\in W_{\eps}(A,\sigma)$ we have $f(x)^2=\tld{\sign}_{P,p}^\eta(x)^2$.
  Since $S$ is integral, this means $f(x)=\tld{\sign}_{P,p}^\eta(x)$
  for some $\eta$ (in particular $S=R$), and the rest of the reasoning is
  also the same, the only difference being that we have to invoque that
  $p\neq 2$ to justify that $\eta$ is uniquely determined (and thus the
  two possible ideals are really distinct).
\end{proof}

\begin{coro}
  Let $(A,\sigma)$ be an algebra with involution over $K$. We have
  \[ \Spec(\tld{W}(A,\sigma)) = \{I(A,\sigma)\} \coprod_{P\in X(K)}
    \ens{I_{P,p}^\eta(A,\sigma)}{\eta=\pm 1,\text{$p$ odd or $0$}}. \]
\end{coro}

\begin{rem}
  The proof also applies to compute the spectrum of $\tld{W}_\eps(A,\sigma)$:
  there is one prime above $I(K)$, two above each $I_{P,p}(K)$ with
  $P\in X_\eps(A)$ and one above each $I_{P,p}(K)$ with $P\in X_{-\eps}(A)$.
\end{rem}

\begin{rem}
  In the continuity of remark \ref{rem_compar_sign_au}, Astier
  and Unger show in \cite[6.5,6.7]{AU2} slightly different
  and arguably stronger results, since they obtain a similar
  classification with a weaker notion of ``ideal'' (which does
  not involve the product of two hermitian forms). There is however a
  difference for primes above $I(K)$, since they find many
  such ``ideals'' (but of course only $I(A,\sigma)$ is
  an actual ideal).
\end{rem}

Emulating the classical case, we set
\[ \tld{X}(A,\sigma)=\Spec_0(\tld{W}(A,\sigma)) \]
as a topological subspace of $\Spec(\tld{W}(A,\sigma))$;
its elements are the $I_P^\eta:= I_{P,0}^\eta$ for $P\in X(K)$.
Thus the continuous map $\pi: \Spec(\tld{W}(A,\sigma))\to \Spec(W(K))$
induces a continuous two-to-one map $\pi:\tld{X}(A,\sigma)\to X(K)$.

\begin{lem}
  The map $\pi: \tld{X}(A,\sigma)\to X(K)$ is open.
\end{lem}

\begin{proof}
  Let $x\in \tld{W}(A,\sigma)$. Since the principal open sets $D(x)$
  form a base of the Zariski topology, it is enough to show that $\pi(U)$
  is open, where $U=D(x)\cap \tld{X}(A,\sigma)$. Let us write $x=x_0+x_1$,
  with $x_0\in W(K)$ and $x_1\in W(A,\sigma)\oplus W^{-1}(A,\sigma)$.

  Then $\pi(U)$ is the set of orderings $P$ such that at least one
  of the $\tld{\sign}_P^\eta(x)$ is non-zero. But the only way they can
  both be zero is if $\sign_P(x_0)$ and $\tld{\sign}_P^\eta(x_1)=0$,
  and that second equality is equivalent to $\sign_P(x_1^2)=0$. Those
  two conditions define a closed subset of $X(K)$, so their complement
  is open.
\end{proof}

As in the classical case we have a total signature:

\begin{defi}
  Let $(A,\sigma)$ be an algebra with involution over $K$. The total
  signature of any $x\in \tld{W}(A,\sigma)$ is the function
  \[ \tld{\sign}(x): \tld{X}(A,\sigma)\To \Z \]
  obtained by reducing $x$ modulo the various primes in $\tld{X}(A,\sigma)$.
\end{defi}

Explicitly, $\tld{\sign}(x)(I_P^\eta) = \tld{\sign}_P^\eta(x)$, but it
is important to note that $\tld{\sign}(x)$ does not depend on any choice
of polarization.

\begin{prop}\label{prop_tot_sign_cont}
  Let $(A,\sigma)$ be an algebra with involution over $K$. Then
  for any $x\in \tld{W}(A,\sigma)$, the total signature $\tld{\sign}(x)$
  is a continuous function.
\end{prop}

\begin{proof}
  By definition, $\tld{\sign}(x)^{-1}(\{n\})$ is the intersection
  of $\tld{X}(A,\sigma)$ with the Zariski-closed set $V(x-n\fdiag{1})$
  in $\Spec(\tld{W}(A,\sigma))$, so it is closed in $\tld{X}(A,\sigma)$.
\end{proof}





  

\section{Polarizations}\label{sec_polar}

One on the main goals in \cite{AU} and \cite{AU2} can be interpreted
as the definition of an appropriate total signature that is defined
on $X(K)$ instead of $\tld{X}(A,\sigma)$ (this is what they call
$\mathcal{M}$-signatures and $H$-signatures).

\begin{defi}
  Let $(A,\sigma)$ be an algebra with involution over $K$.
  If $U$ is an open subset of $X(K)$, a local polarization
  of $(A,\sigma)$ over $U$ is a set-theoretical section of
  $\bar{\pi}$ on $U$. We write $\Pol_U(A,\sigma)$ for the
  set of local polarizations over $U$. If $s\in \Pol_U(A,\sigma)$,
  we say that $-s\in \Pol_U(A,\sigma)$, such that $-s(P)\neq s(P)$
  for all $P\in U$, is the opposite (local) polarization of $s$.

  When $U=X(K)$ (resp. $X_+(A)$, $X_-(A)$), we speak of a
  global (resp. orthogonal, symplectic) polarization of
  $(A,\sigma)$, and the set of those is denoted by $\Pol(A,\sigma)$
  (resp. $\Pol_+(A,\sigma)$, $\Pol_-(A,\sigma)$). A global
  polarization is also simply called a polarization.

  If $s\in \Pol(A,\sigma)$, then for any $x\in \tld{W}(A,\sigma)$,
  the total signature of $x$ relative to $s$ is
  \[ \tld{\sign}^s(x): X(K) \xrightarrow{s} \tld{X}(A,\sigma)\xrightarrow{\tld{\sign}(x)} \Z. \]
  We also write $\tld{\sign}_P^s(x) = \tld{\sign}^s(x)(P)$.
\end{defi}

Clearly $\Pol(A,\sigma)\simeq \Pol_+(A,\sigma)\times \Pol_-(A,\sigma)$.
The way we see things is that a polarization is the choice
of a labelling of $\tld{\sign}_P^+$ and $\tld{\sign}_P^-$,
and an orthogonal (resp. symplectic) polarization is such a
choice for only the $P\in X_+(A)$ (resp. $X_-(A)$).
The way we defined the signature maps shows that a choice
of polarization is also equivalent to a choice of Morita
equivalence between $(A_{K_P},\sigma_{K_P})$ and $(D_P,\theta_P)$
for all $P\in X(K)$, but the global structure of $\tld{X}(A,\sigma)$
makes it much more convenient to discuss polarizations.
Our goal is to find relevant natural classes of polarizations, or
even ideally natural polarizations on various $(A,\sigma)$.

\begin{rem}
  The notion of $\mathcal{M}$-signature in \cite{AU} corresponds to
  an arbitrary (orthogonal/symplectic) polarization.
\end{rem}

\begin{rem}
  For any polarization $s$, and any $x\in W(K)$, $\tld{\sign}^s(x)$
  is the classical total signature $\sign(x): X(K)\to \Z$.
\end{rem}

There are natural symmetries of polarizations that we want
to emphasize. Let $G$ be the group of set-theoretical
automorphisms of $\bar{\pi}$, and $G_c$ (for \emph{continuous})
the group of topological automorphisms of $\bar{\pi}$. Then
$G$ can be identified with the multiplicative group
$\mathcal{F}(X(K),\{\pm 1\})$: a function $f:X(K)\to \{\pm 1\}$
acts by swapping the elements of the fiber above
$P\in X(K)$ iff $f(P)=-1$. This is also naturally
isomorphic to the group $(\mathcal{P}(X(K)),\Delta)$ of
subsets of $X(K)$ with the symmetric difference,
is we associate $f$ to $f^{-1}(\{-1\})$. Then $G_c\subset G$
corresponds to the continuous functions inside
$\mathcal{F}(X(K),\{\pm 1\})$, and to the clopen subsets
inside $\mathcal{P}(X(K))$. Clearly $G$ acts, simply transitively,
on $\Pol(A,\sigma)$. If $f\in \mathcal{F}(X(K),\{\pm 1\})$,
then its action (as an element of $G$) on the total signatures is
$\tld{\sign}^{f\cdot s} = f\tld{\sign}^s$.
The function $s\mapsto -s$ corresponds to the constant function
$-1$ in $\mathcal{F}(X(K),\{\pm 1\})$, and to $X(K)\in \mathcal{P}(X(K))$.

Note that since $\bar{\pi}$ is the application of the functor
$\Spec_0$ to the inclusion $W(K)\to \tld{W}(A,\sigma)$, there
is a canonical embedding of the $W(K)$-algebra automorphisms
of $\tld{W}(A,\sigma)$ in $G_c$. We call $G_a$ (for \emph{algebraic})
the image of the embedding. The image of the subgroup of
standard automorphisms is denoted $G_s$ (for \emph{standard}).
This action can also be deduced from the fact that by
construction, $(A,\sigma)\mapsto \tld{X}(A,\sigma)$
defines a functor from $\CBrh$ to the category of sets
above $X(K)$. Note that $G_s$ is naturally a quotient of
$K^*$ since standard automorphisms have the form $(\fdiag{a}_\sigma)_*$
for some $a\in K^*$.

We then have
\[ G_s\subset G_a\subset G_c\subset G, \]
and if we cannot find a canonical element of $\Pol(A,\sigma)$ for
an arbritrary $(A,\sigma)$ we can at least try to find
canonical classes in $\Pol(A,\sigma)/H$ for those various
subgroups $H\subset G$ (of course $\Pol(A,\sigma)/G=\{\ast\}$),
or maybe at least in $\Pol_\eps(A,\sigma)/H$ for some $\eps$.

\begin{rem}
  Note that by functoriality with respect to $\CBrh$,
  $\Pol(A,\sigma)/G_s$ only depends on the Brauer class $[A]$.
\end{rem}

The topological nature of our spaces makes it very natural to
investigate the following class:

\begin{defi}
  We write $\Pol^c(A,\sigma)$ for the set of continuous polarizations
  on $(A,\sigma)$, that is continuous sections of $\bar{\pi}$.
  Likewise, we have $\Pol^c_+(A,\sigma)$ and $\Pol^c_-(A,\sigma)$,
  so that $\Pol^c(A,\sigma)\simeq \Pol^c_+(A,\sigma)\times \Pol^c_-(A,\sigma)$.
\end{defi}

\begin{rem}
  By construction, a polarization is the same as a set-theoretic section
  of $\pi:\Spec(\tld{W}(A,\sigma))\to \Spec(W(K))$ that is compatible
  with the specialization of points. Then a continuous polarization
  is the same as a continuous section of $\pi$.
\end{rem}

\begin{prop}
  Let $(A,\sigma)$ be an algebra with involution over $K$. A polarization
  $s\in \Pol(A,\sigma)$ is continuous iff for all $x\in \tld{W}(A,\sigma)$,
  the total signature $\tld{\sign}^s(x)$ relative to $s$ is continuous
  on $X(K)$.
\end{prop}

\begin{proof}
  Since the absolute total signature $\tld{\sign}(x)$ is continuous
  on $\tld{X}(A,\sigma)$ (proposition \ref{prop_tot_sign_cont}),
  clearly if $s$ is a continuous section of $\bar{\pi}$ then the
  composition $\tld{\sign}^s(x)$ is also continuous.

  Conversely, assume all $\tld{\sign}^s(x)$ are continuous on
  $\tld{X}(A,\sigma)$. Let $D(x)\subset \Spec(\tld{W}(A,\sigma))$
  be the open subset defined by $x$ (ie the open subscheme defined by
  the localization at $x$). By construction, $D_0(x):=D(x)\cap \tld{X}(A,\sigma)$
  is the subset on which $\tld{\sign}(x)$ takes non-zero values,
  so $s^{-1}(D_0(x))$ is the subset of $X(K)$ on which $\tld{\sign}^s(x)$
  takes non-zero values. By hypothesis, it is open in $X(K)$.
  Since the $D_0(x)$ form an open basis of $\tld{X}(A,\sigma)$,
  this means that $s$ is continuous.
\end{proof}

It follows from the definition of $G_c$ that if $\Pol^c(A,\sigma)$
is not empty, then it is a simply transitive $G_c$-set, so it
defines a class in $\Pol(A,\sigma)/G_c$. Thus we just need to
know whether there is one continuous polarization to find
all of them. This is strongly related to the study of $H$-signatures
in \cite{AU2}, as we will now investigate.

\begin{defi}
  Let $(A,\sigma)$ be an algebra with involution over $K$. If
  $x\in \tld{W}(A,\sigma)$, we write
  \[ U(x) = \ens{P\in X(K)}{\tld{\sign}_P^+(x)\neq \tld{\sign}_P^-(x)}. \]
  We call $U(x)$ the principal subset of $X(K)$ defined by $x$.

  We also define $s_x\in \Pol_{U(x)}(A,\sigma)$, called the principal
  local polarization defined by $x$, as the unique local polarization
  such that $\tld{\sign}^{s_x}_P(x) > \tld{\sign}^{-s_x}_P(x)$ for
  all $P\in U(x)$.
\end{defi}

\begin{prop}
  Let $(A,\sigma)$ be an algebra with involution over $K$ Then for any
  $x\in \tld{W}(A,\sigma)$, $U(x)$ is a clopen subset of $X(K)$,
  and $s_x: U(x)\to \tld{X}(A,\sigma)$ is a continuous local polarization of
  $(A,\sigma)$ over $U(x)$. 
\end{prop}

\begin{proof}
  Let $\tau: \tld{X}(A,\sigma)\to \tld{X}(A,\sigma)$ be the function
  that swaps the elements of every fiber of $\pi$. It is continuous,
  for instance because it is induced by the standard automorphism defined
  by $\fdiag{-1}_\sigma$. We define $f: \tld{X}(A,\sigma)\to \Z^2$
  by $f= (\tld{\sign}(x),\tld{\sign}(x)\circ \tau)$, and
  \[ S = \ens{(m,n)\in \Z^2}{m\neq n},\quad S^+ = \ens{(m,n)\in \Z^2}{m> n}. \]

  Then $U(x) = \pi(f^{-1}(S))$ and $\Ima(s_x)=f^{-1}(S^+)$, so $\Ima(s_x)$
  is closed in $\tld{X}(A,\sigma)$ and $U(x)$ is clopen in $X(K)$ (here
  we use the compacity of $\tld{X}(A,\sigma)$). Now if $Y$ is any
  closed set in $\tld{X}(A,\sigma)$, then $s_x^{-1}(Y) = \pi(Y\cap \Ima(s_x))$,
  so it is closed in $X(K)$, which shows that $s_x$ is continuous.
\end{proof}

Then we interpret the results in \cite{AU2} as:

\begin{thm}\label{thm_polar}
  Let $(A,\sigma)$ be an algebra with involution over $K$. Then for any
  $x_1,\dots,x_n\in W_\eps(A)$, there exists $x\in W_\eps(A,\sigma)$ such that
  $U(x_1)\cup \dots \cup U(x_n) = U(x)$. In particular,
  there exists $x\in W_\eps(A,\sigma)$ such that $U(x)=X_\eps(A)$, so there exist
  global continuous polarizations, and $\tld{X}(A,\sigma)\approx X(K)\coprod X(K)$
  as topological spaces, with $\bar{\pi}$ being the canonical projection.

  Furthermore, the class of global principal polarizations is a transitive
  $G_c$-set, thus it is exactly $\Pol^c(A,\sigma)$.
\end{thm}

\begin{proof}
  The existence of $x\in W_\eps(A,\sigma)$ such that $\bigcup U(x_i) = U(x)$ is a
  reformulation of \cite[3.1]{AU2}. The existence of $x\in W_\eps(A,\sigma)$ such that
  $U(x)=X_\eps(A)$ follows by compacity, since lemma \ref{lem_h_local} shows
  that the $U(x_i)$ form an open cover of $X_\eps(A)$. Since $\Pol^c_+(A,\sigma)$
  and $\Pol^c_-(A,\sigma)$ are non-empty, so is $\Pol^c(A,\sigma)$. If
  $s\in \Pol(A,\sigma)$, then $\tld{X}(A,\sigma) = \Ima(s)\coprod \Ima(-s)$,
  and if $s$ is continuous, $\Ima(s)$ and $\Ima(-s)$ are homeomorphic to $X(K)$.

  The fact that the principal polarizations are a transitive $G_c$-set is a
  reformulation of \cite[3.3]{AU2}. Since they are included in $\Pol^c(A,\sigma)$
  and $\Pol^c(A,\sigma)$ is also a transitive $G_c$-set, we can conclude.
\end{proof}

\begin{rem}
  With this framework, the weaker lemma \ref{lem_h_local} simply states that the
  $U(x)$ with $x\in W_\eps(A,\sigma)$ form a open cover of $X_\eps(A)$, whichs
  shows that $\tld{X}(A,\sigma)$ is a double cover of $X(K)$, since
  it has local trivialization. The notion of $H$-signatures in \cite{AU}
  corresponds to taking a finite open cover of $X_\eps(A)$ by principal subsets
  (which exists by compacity). In general, $H$-signatures correspond to
  continuous polarizations.
\end{rem}

Given that $\Spec(\tld{W}(A,\sigma))$ is not only a topological space
but a scheme, we also have another natural class of polarizations:
if $\rho: \tld{W}(A,\sigma)\to W(K)$ is a retraction (see definition
\ref{def_retraction}), then applying the $\Spec$ functor gives a
scheme morphism $\rho^*: \Spec(W(K))\to \Spec(\tld{W}(A,\sigma))$,
so in particular a continuous polarization. We call polarizations
of this form \emph{algebraic polarizations} of $(A,\sigma)$, and
we write $\Pol^a(A,\sigma)$. Similarly, orthogonal (resp. symplectic)
retractions define the set $\Pol^a_+(A,\sigma)$ (resp. $\Pol^a_-(A,\sigma)$)
of algebraic orthogonal (resp. symplectic) polarizations of $(A,\sigma)$.
There is an obvious injection $\Pol^a(A,\sigma)\subset \Pol^a_+(A,\sigma)\times \Pol^a_-(A,\sigma)$,
but it is not clear if it is surjective in general.
The existence of an algebraic (global, orthogonal or symplectic) polarization
of $(A,\sigma)$ obviously depends only on the Brauer class $[A]$.

\begin{rem}
  By construction, $G_a$ acts on $\Pol^a_+(A,\sigma)$ and $\Pol^a_-(A,\sigma)$,
  and it acts transitively on $\Pol^a(A,\sigma)$. In particular,
  there is a well-defined ``algebraic'' element in $\Pol(A,\sigma)/G_a$,
  and the sets $\Pol^a(A,\sigma)/G_s$, $\Pol_+^a(A,\sigma)/G_s$ and
  $\Pol_-^a(A,\sigma)/G_s$ are well-defined, and they only depend on
  the Brauer class $[A]$. 
\end{rem}

\begin{ex}
  If $A$ is split, there are algebraic polarizations of $(A,\sigma)$,
  by example \ref{ex_retrac_ortho}.
\end{ex}

\begin{ex}
  According to proposition \ref{prop_retrac_sympl}, there is always a canonical
 algebraic symplectic
  polarization of $(\mathbf{H}_K,\gamma)$, and if $K$ is Pythagorean
  there is a canonical global algebraic polarization. On the
  other hand, $(\mathbf{H}_K,\gamma)$ does not have algebraic
  polarizations if the Pythagoras number of $K$ is at least 3.
\end{ex}

We do not know of any other cases where algebraic polarizations
exist, and it would be interesting to characterize the Brauer classes
for which it is the case.

\bibliographystyle{plain}
\bibliography{signature_first_kind}

\end{document}

%% file: canevas_anglais.tex
\usepackage[utf8]{inputenc}
\usepackage[english]{babel}
\usepackage[T1]{fontenc}
\usepackage{amsmath}
\usepackage{amsfonts}
\usepackage{amssymb}
\usepackage{amsthm}
\usepackage{stmaryrd}
\usepackage{tikz}
\usepackage{tikz-cd}
\usepackage{enumitem}
\usepackage{todonotes} 
\usepackage{graphicx}
\usepackage{hyperref}

\DeclareMathOperator{\Spec}{Spec}

\DeclareMathOperator{\End}{End}

\DeclareMathOperator{\Ker}{Ker}

\DeclareMathOperator{\Ima}{Im}

\DeclareMathOperator{\Trd}{Trd}

\DeclareMathOperator{\Id}{Id}

\DeclareMathOperator{\rdim}{rdim}
\DeclareMathOperator{\sign}{sign}
\DeclareMathOperator{\Pol}{Pol}

\newcommand{\isom}{\stackrel{\sim}{\rightarrow}}
\newcommand{\Isom}{\stackrel{\sim}{\longrightarrow}}

\newcommand{\Z}{\mathbb{Z}}
\newcommand{\N}{\mathbb{N}}

\newcommand{\pfis}[1]{\langle\!\langle #1\rangle\!\rangle}
\newcommand{\To}{\longrightarrow}

\newcommand{\fdiag}[1]{\langle #1\rangle}
\newcommand{\ens}[2]{\{ #1\,|\, #2\}}

\newcommand{\tld}{\widetilde}
\newcommand{\eps}{\varepsilon}

\newcommand{\CBrh}[1][K]{\mathbf{Br}_h(#1)}
\newcommand{\Zd}{\Z/2\Z}

\renewcommand{\phi}{\varphi}
\renewcommand{\bar}{\overline}

\newtheorem{thm}{Theorem}[section]
\newtheorem{prop}[thm]{Proposition}
\newtheorem{coro}[thm]{Corollary}
\newtheorem{lem}[thm]{Lemma}
\newtheorem{defi}[thm]{Definition}

\theoremstyle{definition}
\newtheorem{rem}[thm]{Remark}
\newtheorem{ex}[thm]{Example}

%% file: signature_first_kind.bbl
\begin{thebibliography}{1}

\bibitem{AS}
Emil Artin and Otto Schreier.
\newblock Algebraische {Konstruktion} reeller {Körper}.
\newblock {\em Abhandlungen aus dem Mathematischen Seminar der Universität
  Hamburg}, 5(1):85--99, 1927.

\bibitem{AU}
Vincent Astier and Thomas Unger.
\newblock Signatures of hermitian forms and the {Knebusch} {Trace} {Formula}.
\newblock {\em Mathematische Annalen}, 358(3-4):925--947, 2014.
\newblock arXiv: 1003.0956.

\bibitem{AU2}
Vincent Astier and Thomas Unger.
\newblock Signatures of hermitian forms and "prime ideals" of witt groups.
\newblock {\em Advances in Mathematics}, 285:497--514, 2015.
\newblock arXiv: 1303.3494.

\bibitem{BP}
Eva Bayer-Fluckiger and Raman Parimala.
\newblock Classical groups and the hasse principle.
\newblock {\em Annals of mathematics}, 147(3):651--693, 1998.

\bibitem{G}
Nicolas Garrel.
\newblock Mixed witt rings of algebras with involution.
\newblock {\em Canadian Journal of Mathematics}, 75(2):608–644, 2023.

\bibitem{BOI}
Max-Albert Knus, Alexander Merkurjev, Markus Rost, and Jean-Pierre Tignol.
\newblock {\em The {Book} of {Involutions}}.
\newblock American Mathematical Soc., 1998.

\bibitem{Lam}
T.Y. Lam.
\newblock {\em Introduction to Quadratic Forms over Fields}.
\newblock AMS, 2005.

\bibitem{LT}
David~W. Lewis and J.~P. Tignol.
\newblock On the signature of an involution.
\newblock {\em Archiv der Mathematik}, 60(2):128--135, 1993.

\bibitem{AQM}
Anne Quéguiner.
\newblock Signature des {Involutions} de deuxième espèce.
\newblock {\em Archiv der Mathematik}, 65(5):408--412, 1995.

\end{thebibliography}
